\documentclass{amsart}

\usepackage{
amsfonts,
amsmath,
latexsym,
amssymb,
enumerate,
verbatim,
mathrsfs,
graphicx,
centernot,
accents, 
color,
}

\usepackage[all]{xy}

\newcommand{\labbel}[1]{\label{#1} [[{\bf #1}]]}  
\renewcommand{\labbel}{\label}

\newcommand{\ddd}{\mathrel{\delta}} 
\newcommand{\nd}{\mathrel{\centernot\delta}}

\newtheorem{theorem}{Theorem}[section]
\newtheorem{lemma}[theorem]{Lemma}

\newtheorem{proposition}[theorem]{Proposition} 
 
\newtheorem{corollary}[theorem]{Corollary} 
 
\newtheorem{claim}[theorem]{Claim} 
\newtheorem*{claim*}{Claim}

\newtheorem*{theorem*}{Theorem}
\newtheorem*{proposition*}{Proposition}
\newtheorem*{corollary*}{Corollary}
\newtheorem*{lemma*}{Lemma}
\newtheorem*{scholion*}{Scholion}

\theoremstyle{definition}

\theoremstyle{remark}
\newtheorem{remark}[theorem]{Remark}

\newtheorem*{remark*}{Remark}
\newtheorem*{remarks*}{Remarks}

\newtheorem*{acknowledgement}{Acknowledgement}

\newtheorem*{observation*}{Observation}

 \allowdisplaybreaks[1]

\begin{document}

 \title{Ivanova contact join-semilattices are not finitely axiomatizable}

\author{Paolo Lipparini} 
\address{Dipartimento di Matematica\\Viale 
dei Ricontatti di  Cerca  
\\Universit\`a di Roma ``Tor Vergata'' 
\\I-00133 ROME ITALY}

\email{lipparin@axp.mat.uniroma2.it}


\begin{abstract}
We show that the class of contact join-semilattices
 introduced by T. Ivanova
(Studia Logica
\textbf{110}, 1219--1241, 2022)
 is not finitely axiomatizable.
On the other hand, a simple finite axiomatization
exists for the class of those join-semilattices with a weak contact relation
 which can be embedded into 
the reduct of 
a weak contact Boolean algebra 
(equivalently, 
distributive lattice).
\end{abstract} 


\maketitle

\section{Introduction} \labbel{intro}

Structures with a proximity  or a contact  relation are useful
in topology \cite{DC}, algebraic logic \cite{BCGL,C}, 
computer science \cite{WN}\footnote{See
\cite[Remark 6.11]{hypercs} 
 for explanations about the terminology.},
image analysis \cite{PN}, knowledge representation
\cite{V}, graph theory \cite{BC,hypercs} 
  and are a fundamental tool
in region based theory of space \cite{BD}.
Recall that a  contact algebra is a Boolean algebra endowed with a further
contact relation.
D{\"u}ntsch,  MacCaull, Vakarelov, Winter
 \cite{DMVW} proposed to study
contact relations in algebras with less structure; 
in particular, Ivanova \cite{I} suggested the naturalness
of studying join-semilattices with a contact. See also \cite{mtt}
for further arguments in favor of the use of the semilattice
operation only. We refer to the quoted sources 
and to \cite{cs} for further
details, motivation, history and references.
 
Ivanova \cite{I} provided an axiomatization for those
contact join-semilattices which admit a representation as a
substructure of some
field of sets, where two sets are in  contact if
they have nonempty intersection.
This is  equivalent to being embeddable into
a complete and atomic Boolean algebra  with overlap contact.
Ivanova also showed that the axiomatization
characterizes contact join-semilattices
 which admit an appropriate topological representation,
and asked whether there is a finite axiomatization.
In \cite{cs} we presented an equivalent axiomatization
for Ivanova contact join-semilattices and also axiomatized the class
of those weak contact join-semilattices embeddable into 
a, possibly complete and atomic, 
 weak contact Boolean algebra, equivalently, 
into a weak contact  distributive lattice
(here the contact is not necessarily overlap).

Here we use the axiomatizations from \cite{cs}  in order to show
that the former class  is not
finitely axiomatizable, while the latter is.
As a consequence, from \cite[Corollary 5.1]{cs}
we get that in the language of
contact join-semilattices the set of logical consequences of the theory of Boolean 
algebras with an overlap contact relation is not finitely
axiomatizable. On the other hand, an easily described finite axiomatization exists
if we remove the request that the contact relation is overlap.
We also show that the validity of some of the representation theorems
mentioned above is equivalent to the Boolean Prime Ideal Theorem,
hence needs a consequence of the axiom of choice. See Proposition \ref{zf}.

As we mentioned, motivations for the study of join-semilattices
with a contact relation are presented in
\cite{C,DMVW,I,cs}, among others.
Our main technique in the proof of 
Theorem \ref{dn} is constructive and 
is likely to be useful for constructing many more
interesting examples of contact join-semilattices. 
A further possible  research direction
is the study of join-semilattices with hypercontact
$n$-ary relations, rather than just a binary contact.
See \cite{V}. In \cite{hypercs} we  show that a large part of 
the results presented here generalize to the wider context.   
See Remark \ref{final} below for a discussion.

\section{Preliminaries} \labbel{prel} 

For simplicity, we will always assume that posets 
and semilattices have a minimum element $0$ and that homomorphisms
preserve $0$. 
A \emph{weak contact relation} 
on some poset $\mathbf S$ 
with $0$ is 
 a symmetric reflexive binary relation $\delta$  on $S \setminus \{ 0 \} $
 such that     
\begin{align}
\labbel{ext}    \tag{Ext} 
& a \ddd b \ \  \& \ \  a \leq a_1\ \  \&\ \  b \leq b_1 \ 
\Rightarrow \  a_1 \ddd b_1,
  \end{align}
for all $ a, b, a_1, b_1 \in S$. 
We write $a \nd b$ to mean that  $a \ddd b$ does not
hold.  

A typical example of a weak contact relation
is the \emph{overlap} relation, which can be defined on any poset
$\mathbf P$ 
with $0$. To get the overlap relation, simply set
$a \ddd b$ if there is $p \in P$ such that $p >0$, $p\leq a$
and $p \leq b$.  
 
A \emph{weak contact  join-semilattice}
is a structure   $(S, {+}, 0, { \delta  } )$,
where $(S, {+}, 0)$ is a join-semilattice with $0$  and $\delta$ 
is a   weak contact relation, as defined above.
Note that in some previous papers we frequently used
the word  ``semilattice'' in order to mean ``join-semilattice''.
More generally, a  \emph{weak contact  lattice} is a lattice
together with a weak contact relation. Weak contact
Boolean algebras are defined similarly.
 
We will occasionally also deal with nonsymmetrical 
relations.
 A \emph{weak pre-contact relation} 
on some poset $\mathbf S$ 
with $0$ is 
 a reflexive binary relation $\delta$  on $S \setminus \{ 0 \} $
satisfying \eqref{ext}. Thus in a weak pre-contact
we leave out the request that $\delta$ is symmetrical.
 Note that some authors do not assume  reflexivity
in the definition of a weak pre-contact relation,
but here we will always assume reflexivity.

Homomorphisms and embeddings
are always intended in the usual sense.
For example, a \emph{homomorphism} 
 of weak (pre-)contact  join-semilattices
is a $0$-preserving semilattice homomorphism $\varphi$  such that 
(i) $a \ddd b$ implies $ \varphi (a) \ddd \varphi (b)$,
for all $a,b$ in the domain.
An \emph{embedding}  of    weak contact  join-semilattices
 is an injective homomorphism such that
the reverse condition also holds, namely,
(ii) $ \varphi (a) \ddd \varphi (b)$  implies $a \ddd b$.
In certain situations we need to assume that $\varphi$  preserves
only part of the structure, not necessarily all the structure; for example, a
\emph{$\delta$-homomorphism}, or a \emph{contact homomorphism} 
is a function $\varphi$   such that (i) above holds,
but $\varphi$  is not necessarily required to be, say, order preserving.

The following property
is frequently required in the definition of a contact
relation on a join-semilattice (this is the reason for the terminology including
``weak'').
An \emph{additive contact relation} on some join-semilattice is a weak contact relation
satisfying the following condition. 
\begin{align} 
\labbel{add}    \tag{Add} 
&  a \ddd b{+}c 
\Rightarrow 
 a \ddd b \text{ or } a \ddd c.
 \end{align} 
We will not assume additivity, unless explicitly mentioned.

As mentioned in the introduction,
Ivanova \cite{I} axiomatized those contact join-semilattices
which admit   good set-theoretical and topological representations, 
calling them  ``contact join-semilattices''.
Here, when we mention a (weak) contact  join-semilattice
we will only assume the semilattice axioms, together with the properties
of (weak) contact listed above. Hence
we will refer to the structures considered in 
\cite{I} as  \emph{Ivanova contact join-semilattices}. 
In \cite[Corollary 3.4]{cs} we have provided
an alternative axiomatization of Ivanova contact join-semilattices, see below 
for details.
In \cite{I} semilattices are assumed to have a maximum $1$.
As far as the results presented here are concerned, it is irrelevant 
whether we request or
not the existence of $1$. See \cite[Remark 3.3(a)]{cs}.    

Now we list the relevant conditions,
where $\mathbf S$ is a weak contact join-semilattice and
$n$ varies among positive natural numbers.

\begin{align}
\labbel{d1+}    \tag{D1+} 
&\begin{aligned} 
& \text{For every } n  \text{ and } 
a,b ,  c_{1,0}, c_{1,1},  
\dots, c_{n,0},  c_{n,1} \in S,
\\
& \text{IF } 
  c_{1,0}\nd c_{1,1}, \  c_{2,0}\nd c_{2,1}, \  
\dots, \  c_{n,0}\nd c_{n,1} 
\\
&\text{and  }
b \leq 
a +  c_{1,f(1)}  + \dots + 
 c_{n,f(n)}, \text{ for all $f: \{ 1, \dots , n\} \to \{ 0, 1  \}$,}
\\
& \text{THEN  } b \leq a.  
\end{aligned} 
\\[8 pt]
 \labbel{d2n}    \tag{D2$_n$} 
&\begin{aligned} 
& \text{For every  } 
a,b ,  c_{1,0}, c_{1,1},  
\dots, c_{n,0},  c_{n,1} \in S,
\\
& \text{IF } 
  c_{1,0}\nd c_{1,1},   
\dots,   c_{n,0}\nd c_{n,1} 
\text{ and, for every $f: \{ 1, \dots , n\} \to \{ 0, 1  \}$,}
\\
&\text{at least one of the following two inequalities holds}
\\
&b \leq   c_{1,f(1)}  + \dots + c_{n,f(n)}, \text{ or } 
a \leq   c_{1,f(1)}  + \dots + c_{n,f(n)},
\\
& \text{THEN  } b \nd a.  
\end{aligned} 
  \end{align}  
 
The special case $n=1$ of 
\eqref{d1+} has been called (D1)  in \cite{cs}, where   
we have proved that  (D1)  implies \eqref{d1+}.
The assumption that \eqref{d2n} holds for every 
$n$ has been called (D2) in  \cite[Section 2]{cs}.
In Section \ref{nofin} here we will show that, contrary to the case of \eqref{d1+},
(D2) is not equivalent to a first-order sentence.  

\begin{theorem} \labbel{ivan}
\cite[Corollary 3.4]{cs} 
 A weak  contact join-semilattice $\mathbf S$
is an Ivanova contact join-semilattice if and only if
$\mathbf S$ satisfies  \textup{(D1)}
and \eqref{d2n}, for every   
 positive integer $n$. 
 \end{theorem} 

As far as the present note is concerned, the reader might take
the characterization in Theorem  \ref{ivan} 
 as the definition of an  Ivanova contact join-semilattice.
We refer to \cite{I,cs} for further details about
the above notions. 

The next useful lemma is implicit in the proof of \cite[Theorem 4.1]{cs}.

\begin{lemma} \labbel{lemcombis}
Suppose that 
$\mathbf S = (S, {\leq}, 0 , \delta_S ) $ is a poset with a
weak  (pre-)contact relation $\delta_S$,
$\mathbf Q = (Q, {\leq}, 0  ) $ is a poset with $0$ and $\kappa$ 
is a function from $\mathbf S$ to
$\mathbf Q$ such that  
$\kappa(a) = 0$ implies $a= 0$, for every $a \in S$.
 Let $\delta_Q$ be defined on $\mathbf Q$ by letting
$ b_1 \ddd_Q b_2  $
if at least one of the following conditions hold:

(a) there is $q \in Q$ such that
$0 < q$,  
 $q \leq b_1 $ and  $ q \leq  b_2 $, or 

(b) there are 
$a_1, a_2 \in S$  
 such that $ a_1 \ddd_S a_2  $ and  
such that  $ \kappa  (a_1) \leq b_1$
and $ \kappa  (a_2) \leq b_2$.

 Then
  \begin{enumerate}[(i)]   
\item
$\ddd_Q $ is a weak (pre-)contact on $\mathbf Q$
and $\kappa$ is a $\delta$-homomorphism from 
$\mathbf S$ to $\mathbf Q$. 
In fact, $\ddd_Q $ is the smallest weak (pre-)contact on $\mathbf Q$ 
which makes $\kappa$  a $\delta$-homomorphism.
\item
Suppose in addition that $\kappa$ is an order embedding such that,
whenever $ a_1 \nd_S a_2  $,
the meet of 
$\kappa (a_1) $ and of $  \kappa (a_2)$
 in $\mathbf Q$ exists and is equal to $0$.
Then  $\kappa$ is a $\delta$-embedding from 
$\mathbf S$ to $\mathbf Q$.
    \end{enumerate} 
 \end{lemma} 

\begin{proof}
(i) Symmetry (in the case of a contact relation),
 reflexivity and \eqref{ext} 
for  $ \delta _Q $ are immediate. 
By assumption, if $a \neq 0$,
then $ \kappa (a) \neq 0$, thus all
$ \delta _Q $-related elements are nonzero, since
all
$ \delta _S$-related elements are nonzero.
Thus $ \delta _Q$ is a weak (pre-)contact on $\mathbf Q$;
 $\kappa$ is a $\delta$-homomorphism
by construction.
Every weak (pre-)contact on $\mathbf Q$ must contain
all  the pairs for which (a) holds, by reflexivity
and \eqref{ext}.
If $\kappa$ is a $\delta$-homomorphism 
from $(S, \delta_S) $ to $(Q, \delta'_Q) $ 
and   
$ a_1  \ddd _S a_2 $,
then necessarily $ \kappa (a_1) \ddd' _Q \kappa (a_2)$.
If (b) holds for $ b_1 $ and $  b_2 $ with respect to $ a_1  $ and $ a_2 $
as above, and $ \delta'_Q$ satisfies \eqref{ext},  
 then $ b_1 \ddd'_Q b_2 $. 
Thus $\delta_Q $ is the smallest weak (pre-)contact on $\mathbf Q$
with the required property.

(ii) Because of (i),
we just need to check that if
$ c_1 \nd_S  c_2$, then 
$ \kappa (c_1) \nd _Q  \kappa  (c_2) $.
By assumption,
$ \kappa (c_1) \kappa  (c_2) =0$,
hence (a) cannot be applied in order to get
$ \kappa (c_1) \ddd _Q  \kappa  (c_2) $.
If (b) were applicable,
 there would be 
$a_1, a_2 \in  S$ 
 such that $ a_1 \ddd_S a_2 $ and
  $ \kappa  (a_1) \leq \kappa ( c_1)$,
 $ \kappa  (a_2) \leq \kappa (c_2)$.
Since $\kappa$ is an order embedding, 
$ a_1 \leq  c_1$ and 
$ a_2 \leq  c_2$.
Thus $ c_1 \ddd_S  c_2$, by \eqref{ext}, a contradiction.
 \end{proof}

\section{A finite axiomatization in the non overlap case} \labbel{finite} 

In this section we  present some slight improvements on
Theorems 3.2 and 4.1
from \cite{cs}.
Moreover, we extend \cite[Theorem 4.1]{cs} 
to the case of not necessarily symmetric relations.
We  also show that a consequence of the axiom of
choice is necessary in some results from \cite{cs}.

In what follows we will
frequently  deal with
the situation in which weak (pre-)contact join-semilattices are embedded
into models with further structure, e.~g.,
 distributive lattices or  Boolean algebras.
Rather than explicitly saying that a weak contact join-semilattice $\mathbf S$  can
be embedded into \emph{the contact  join-semilattice reduct} of  some
contact Boolean algebra $\mathbf  B$,  
  we will simply say that  $\mathbf S$  can be 
\emph{$ \{\delta, {+}\}$-embedded}  into 
$\mathbf  B$.
Notice that, on the other hand, we are not assuming
that embeddings preserve existing meets, or complements, unless
otherwise specified.

\begin{theorem} \labbel{thmb}
If $\mathbf S$ is a weak contact join-semilattice, then the following conditions are equivalent, where embeddings are always intended as 
$ \{\delta, {+}\}$-embeddings.
 \begin{enumerate}
   \item
$\mathbf S$  can be embedded into 
a weak contact Boolean algebra.
\item
$\mathbf S$  can be embedded into 
a weak contact distributive lattice.
\item
$\mathbf S$ satisfies \textup{(D1)}.
  \item
$\mathbf S$  can be embedded into 
a weak contact complete atomic Boolean algebra.  
 \end{enumerate}  

The above equivalences hold if
``weak contact'' is 
 everywhere replaced by ``weak pre-contact''.
\end{theorem} 

\begin{proof}
For contact relations, the theorem has been proved in \cite[Theorem 4.1]{cs} 
with an additional assumption  in clause (3).
Hence it is enough to prove that (D1) implies the additional
assumption, which we reproduce here, relabeling some
variables in order to avoid a notational clash in the proof.
\begin{align} \labbel{d2-}    \tag{D2-} 
&\begin{aligned} 
& \text{For every positive } n \in \mathbb N \text{ and } 
a^*,b^*, c_{0,0}, c_{0,1}, c_{1,0}, c_{1,1},  \dots, c_{n,0},  c_{n,1} \in S,
\\
& \text{IF } c_{0,0} \nd c_{0,1},  c_{1,0}\nd c_{1,1},   \dots,   c_{n,0}\nd c_{n,1} 
\text{ and, for every $f: \{ 1, \dots , n\} \to \{ 0, 1  \}$,}
\\
&\text{both }
a^* \leq   c_{0,0}+ c_{1,f(1)}  + \dots + c_{n,f(n)} \text{ and } 
b^* \leq  c_{0,1}+  c_{1,f(1)}  + \dots + c_{n,f(n)},
\\
& \text{THEN  } a^* \nd b^*.  
\end{aligned} 
   \end{align}  
In \cite[Lemma 2.3]{cs}
we have showed that  (D1)  implies \eqref{d1+},
hence, assuming the premises of \eqref{d2-}, 
we get 
$b ^* \leq   c_{0,1}   $ by taking $c_{0,1} $
in place of $a$ in   \eqref{d1+}. Similarly, 
 $a^* \leq   c_{0,0}   $. Since, by the assumptions in (D2-),
$c_{0,0} \nd c_{0,1}$, we get $a^* \nd b^*$ by   \eqref{ext}. 

We now prove the last statement.
The implications (4) $\Rightarrow $  (1) $\Rightarrow $  (2)
are straightforward.
The remaining implications can be obtained by analyzing the proof 
for the case of a weak contact, observing that we have
 not actually used symmetry of
$\delta$ in the above argument and in \cite[Theorem 4.1]{cs}. 
Anyway, we are going to show that the result for a pre-contact
follows from the already proved result for a contact.
We will use the following claim, whose proof is
elementary.

\begin{claim} \labbel{cl}  
If $\delta$ is a weak pre-contact relation on some poset $\mathbf P$,
define $\delta^s$ on $P$ by
$a \ddd^s b$ if   both $a \ddd b$ and $b \ddd a$.
Then  $\delta^s$ is a weak contact relation on $\mathbf P$.
Moreover, $\delta$ satisfies \textup{(D1)}  if and only if 
$\delta^s$ satisfies \textup{(D1)}.
\end{claim}    

Now we can prove (2) $\Rightarrow $  (3) in the case of a weak pre-contact.
Suppose that the weak pre-contact join-semilattice
$\mathbf S$  can be embedded into 
a distributive lattice $\mathbf L$ with
weak pre-contact $\delta_L$.
If we define $\delta_L^s$ as in the Claim, then, by the Claim,
 $\delta_L^s$ is a weak contact relation on $\mathbf L$, thus, by
the theorem in the case of a weak contact, $\mathbf L$ with
$\delta_L^s$ satisfies (D1). By the last statement in the Claim,
$\mathbf L$ with
$\delta_L$ satisfies (D1).
Property (D1) is clearly preserved under substructures
and isomorphism, hence $\mathbf S$ satisfies (D1), as well.
   
To prove  (3) $\Rightarrow $  (4),
assume that  the join-semilattice
$\mathbf S$  with weak pre-contact $\delta$ satisfies (D1).
If we replace $\delta$ with $\delta^s$,
we get a  weak contact join-semilattice
$\mathbf S^s$  satisfying (D1), by the Claim,
hence $\mathbf S^s$ can be embedded
into 
a weak contact complete atomic Boolean algebra
$\mathbf Q$, by the already
proved version of the theorem, dealing with weak contact relations.
If $\kappa$ is the above embedding, then 
$\kappa$ satisfies the assumptions in (ii) in Lemma \ref{lemcombis} 
with respect to $\delta$. Indeed, if $a_1 \nd a_2$,
then also  $a_1 \nd^s a_2$, by definition, hence 
$ \kappa (a_1) \nd_Q \kappa (a_2)$, where 
$\delta_Q$ is the weak contact on $\mathbf Q$.
Thus the meet of   $ \kappa (a_1)  $ and $  \kappa (a_2)$
in $\mathbf Q$ is $0$, by \eqref{ext}.

By   Lemma \ref{lemcombis}(ii)
we can thus endow the Boolean reduct of $\mathbf Q$
with a weak pre-contact $\delta^p$ in such a way that $\kappa$
is a $\delta$-embedding from $\mathbf S$ with $\delta$ 
to $\mathbf Q$
with  $\delta^p$.
The conclusion follows, since the join-semilattice structures
on $\mathbf S$ and $\mathbf Q$ are not affected by the
additional
 contact structure, and $\kappa$ was assumed to be
a semilattice embedding.  
\end{proof}

Since (D1), which is  the instance $n=1$ of \eqref{d1+}, is clearly expressible
as a first-order sentence, we immediately get the following corollary.  

\begin{corollary} \labbel{fa}
The classes of  weak contact and weak pre-contact
 join-semilattices
satisfying one of (equivalently, all of) 
the conditions in Theorem \ref{thmb}
are finitely first-order axiomatizable. 
 \end{corollary} 

In the next section we will show that 
no finite axiomatization exists 
for the class of weak contact join-semilattices
embeddable in a contact Boolean algebra with
\emph{overlap} contact relation, which
by \cite[Theorem 3.2 and Corollary 3.4]{cs}
 is precisely the class of Ivanova
contact join-semilattices.

We now present a slight extension of \cite[Theorem 3.2]{cs}, 
giving a few further characterizations of Ivanova
contact join-semilattices.

If $A$ is a set, $\mathcal P(A)$ denotes the power set of $A$. 
If $X$ is a topological space with closure $K$, the \emph{elementary 
proximity on $X$} \cite{DC} 
is the contact relation  on  $\mathcal P (X)$
 defined by $a \ddd_K b$
if $Ka\cap Kb \neq \emptyset $, for $a, b \subseteq X$.
A contact semilattice is an \emph{elementary topological
contact semilattice} if it has the form 
$(\mathcal P(X), {\cup}, \emptyset,  \delta_K ) $ with $\delta_K$  
as above, for some topological space $X$.  
Note that the above notion of contact is distinct from
another topological notion of contact which is mostly used 
in region based theory of space, and 
which  is defined on the algebra
of regular open, or regular closed subsets of $X$.
See e.~g.  \cite[p. 1221]{I}. 

Recall that a \emph{pre-closure operation} $K$ on some poset
$\mathbf P$ is a unary,  extensive and isotone operation
(the ``pre'' here bears no connection with the ``pre''  
in pre-contact).
If $\mathbf P$ has a minimum element $0$, we will also require that
$K$ is \emph{normal}, that is, $K0=0$.   
If $K$ is also idempotent, it is called a \emph{closure operation}. 
If $\mathbf P$ is a poset with $0$ and with a normal pre-closure
operation $K$, the \emph{associated elementary weak contact relation} 
is defined by
 $a \ddd_K b$ if there is $p \in P$, $p >0$ such that both
 $p \leq Ka$ and $p \leq Kb$. In the above situation, we will say that 
$(P, {\leq} , 0, \delta_K )$ is the
  \emph{elementary contact poset associated to} $\mathbf P$.
We define similar notions for join-semilattices. In the case of semilattices,
a closure operation $K$ is \emph{additive} if $K(x+y)=Kx+Ky$
holds identically.   Recall that a $\delta$-embedding
between structures endowed with a weak contact relation
is an injective function $\varphi$  such that 
$a \ddd b$ if and only if $ \varphi (a) \ddd \varphi (b)$,
for all $a, b$ in the domain.   

\begin{theorem} \labbel{thm}
If $\mathbf S$ is a weak contact join-semilattice, then the following conditions are equivalent, where embeddings are always intended as 
$ \{\delta, {+}\}$-embeddings.
 \begin{enumerate}
\item
$\mathbf S$  can be embedded into 
a Boolean algebra with overlap contact.  
  \item
$\mathbf S$  can be embedded into 
 an additive contact distributive lattice.  
  \item
$\mathbf S$  is an Ivanova contact join-semilattice. 
\item
$\mathbf S$  can be embedded into 
a complete atomic Boolean algebra with overlap contact.  
\item
$\mathbf S$  can be embedded into 
an elementary topological
contact semilattice.
\item
$\mathbf S$  can be embedded into 
the elementary contact join-semilattice associated to some distributive
lattice with additive closure.   
\end{enumerate}  
\end{theorem} 

\begin{proof}
The equivalences of (1) - (4)
follow from \cite{I}; see also \cite[Theorem 3.2]{cs}.

(4) $\Rightarrow $  (5) Since
a complete atomic Boolean algebra $\mathbf  B$ is isomorphic
to a field of sets, say, $\mathcal P(X)$,
if we give $X$ the discrete topology,
the overlap contact on $\mathbf  B$ is the same as the 
elementary proximity on the topological space  $X$.

(5) $\Rightarrow $  (6)  $\Rightarrow $  (2) are elementary.
Indeed, in a distributive lattice with an
additive closure operation $K$, the associated
elementary contact relation is additive.  
See the next lemma for a slightly more general fact.
\end{proof}

Recall that a  lattice with $0$ is \emph{meet semidistributive at $0$} 
if, for all elements $p,q,r$,  $pr=0$ and $qr=0$ imply $(p+q)r=0$.
In particular, a distributive lattice with $0$
is meet semidistributive at $0$ (actually, meet semidistributive, but
we will not use this intermediate notion here).
More generally, let us say that a join-semilattice with $0$
is  \emph{$2$-semidistributive at $0$} 
if, whenever the meets of 
$p , r$ and 
of $q, r $
both exist and are equal to    $0$,
the meet of 
$p+q,  r $  exists and is equal to    $0$.

\begin{lemma} \labbel{semid}
Suppose that  $\mathbf S$ is a join-semilattice with $0$ and 
with a normal additive pre-closure $K$.
If $\mathbf S$ is $2$-semidistributive at $0$,
then the elementary contact associated to $K$
 is additive.
 \end{lemma} 

\begin{proof}
 If $a \nd_K b$   and
 $a \nd_K c$ 
then the meets of 
$Ka, Kb $ and 
of $Ka, Kc $
 exist and are equal to  $0$, by the definition of $ \delta _K$. 
By $2$-semidistributivity at $0$, 
the meet of 
$Ka $ and $Kb+Kc $  exists and is equal to    $0$.
Since $K$ is additive, $Kb+Kc=K(b+c)$ and again the definition 
of $ \delta _K$ gives   $a \nd_K(b+ c)$.
 \end{proof}    
 
Note that a generalization of Theorem \ref{thm},
as it stands, is not possible for weak pre-contact relations,
since the overlap relation is necessarily symmetrical.
It is an open problem to characterize  weak pre-contact join
semilattices satisfying Condition 2 in \ref{thm}, where ``contact''
is replaced by ``pre-contact''.   As far as Clauses (5), (6)
are concerned, note that, given a poset $\mathbf P$ with a pre-closure
operation, we obtain a weak pre-contact relation $\delta$ 
by setting $ a \ddd b$ if there is $p \in P$, $p>0$ such that    
$p \leq a$ and $p \leq Kb$. Again, it is an open problem
to characterize those  weak pre-contact join
semilattices which are representable in such a way, possibly,
when $\mathbf P$  is a distributive lattice and $K$ is an additive
closure. Note also that the construction in Claim \ref{cl}
does not necessarily preserve additivity (when dealing with
pre-contact relations, \emph{additivity} means that both \eqref{add} 
and its symmetric version hold). Indeed, on a finite Boolean algebra,
an additive pre-contact is determined by those atoms which are 
in pre-contact. Thus, if we consider the $8$-element Boolean algebra $\mathbf  B_8$ 
with atoms $a$, $b$ and $c$, the conditions $ a \ddd b$ and
$c \ddd a$  uniquely determine an additive pre-contact on $B_8$. 
However, if $\delta^s$ is defined as in Claim \ref{cl}, then
$\delta^s$ is not additive. Indeed, $a \ddd^s b{+}c$, but
neither $a \ddd^s b$, nor  $a \ddd^s c$.     

As we pointed out in \cite{cs},
we do not need the axiom of choice 
in order to prove the equivalences of (1) - (3)
in Theorems \ref{thmb} and \ref{thm}. 
However, in both proofs we needed the Stone Representation Theorem, 
which is equivalent to the Prime Ideal Theorem \cite[Form 14]{HR},
in order to get the equivalence with (4). 
In the next proposition we point out
that in both cases the equivalence of (1) and (4)
is indeed equivalent to the  Prime Ideal Theorem.

\begin{proposition} \labbel{zf}
In ZF, the Zermelo-Fraenkel theory without the axiom of choice,
the following statements are equivalent.
 \begin{enumerate}[(A)]   
\item 
The  Prime Ideal Theorem \cite[Form 14]{HR}.
\item
The implication (1) $\Rightarrow $  (4) in Theorem \ref{thmb} holds.
\item
The implication (1) $\Rightarrow $  (4) in  Theorem \ref{thm} holds. 
  \end{enumerate} 
 \end{proposition} 

 \begin{proof}
We needed only the 
 Prime Ideal Theorem in the proofs of (1) $\Rightarrow $  (4)
in \cite[Theorems 3.2]{cs} and \cite[Theorems 4.1]{cs}, improved here
in Theorems \ref{thm} and  \ref{thmb}, respectively.
Hence (A) implies both (B) and (C).

Suppose that (B) holds and let $\mathbf  C$ be a Boolean algebra.
Endow $\mathbf  C$ with the overlap contact.
If  the implication (1) $\Rightarrow $  (4) in Theorem \ref{thmb} holds,
then $\mathbf  C$ can be $\{ {  \delta }, {+} \}$-embedded
into some     weak contact complete atomic Boolean algebra $\mathbf  D$.
We check that this embedding, call it $\chi$,  is also a Boolean embedding.
Indeed, if $c \in C$ and $c'$ is the complement of 
$c$, then $c \nd_C c'$, since $\delta_C$ is overlap,
hence $\chi(c) \nd_D  \chi(c') $, since   $\chi$ is 
a $\delta$ embedding. By reflexivity of $\delta_D$ and \eqref{ext},
 $\chi(c)   \chi(c') =0 $; moreover, 
$\chi(c) + \chi(c')=1 $, since $\chi$ is a join-semilattice 
homomorphism. Hence $\chi(c') $ is the complement 
of $  \chi(c) $ in $\mathbf  D$, that is, $\chi$  
 preserves complementation.
By De Morgan law, meet is expressible in terms of join
and complementation, hence  
$\chi$ is a Boolean 
 homomorphism.
Taking the Boolean reduct of $\mathbf  D$,
we get an embedding of $\mathbf  C$ into a complete atomic Boolean algebra.
We have proved the Stone Representation Theorem, 
which is equivalent to the Prime Ideal Theorem \cite[Form 14]{HR},
hence (B) $\Rightarrow $  (A) follows.

Now note that $\mathbf  C$
in the above argument has indeed overlap contact by
construction, hence the argument provides also
(C) $\Rightarrow $  (A).
 \end{proof}

\section{Ivanova contact join-semilattices are 
not finitely axiomatizable} \labbel{nofin}

\begin{theorem} \labbel{dn}
For every $n \geq 2$,
there is a contact join-semilattice satisfying 
\eqref{d1+} and \textup{(D2$_{m}$)},
for every $m<n$,  but not satisfying
\eqref{d2n}. 
 \end{theorem}

\begin{proof} 
Fix $n \geq 2 $ and let $\mathbf  C$ be the Boolean algebra 
freely generated by $n$ generators
$ c_{1,0}, c_{2,0},  
\dots, \allowbreak  c_{n,0}$ and let 
$ c_{1,1} = c' _{1,0}$, \dots, 
$ c_{n,1} = c' _{n,0}$.  Recall that 
a prime denotes Boolean complement.
By construction,  the elements 
$ c_{1,0}, c_{2,0},  
\dots, \allowbreak  c_{n,0}, c_{1,1}, c_{2,1}, \dots, 
 c_{n,1}$ are pairwise incomparable
 with
respect to the standard Boolean order.
Let
\begin{align*} 
\bar  a &= \prod \{ \,    c_{1,f(1)}  + \dots + c_{n,f(n)}  \mid
f: \{ 1, \dots , n\} \to \{ 0, 1  \}, 
  f(1) +  \dots + f(n) \text{ is odd}  \,\},
\\
 \bar b &= \prod \{ \,    c_{1,f(1)}  + \dots + c_{n,f(n)}  \mid
f: \{ 1, \dots , n\} \to \{ 0, 1  \}, 
  f(1) + \dots + f(n) \text{ is even}  \,\}.
\end{align*}
\begin{claim*}
$ \bar{b}$ is the complement of $ \bar{a}$ in $\mathbf  C$, hence,
if $n$ is odd 
\begin{align*} 
\bar  a &= \sum \{ \,    c_{1,f(1)} c_{2,f(2)}   \dots  c_{n,f(n)}  \mid
f: \{ 1, \dots , n\} \to \{ 0, 1  \}, 
  f(1) +  \dots + f(n) \text{ is odd}  \,\},
\\
 \bar b &= \sum \{ \,    c_{1,f(1)}c_{2,f(2)}   \dots  c_{n,f(n)}  \mid
f: \{ 1, \dots , n\} \to \{ 0, 1  \}, 
  f(1) + \dots + f(n) \text{ is even}  \,\},
\end{align*} 
and if $n$ is even the above two displayed identities hold with
$ \bar{a}$ and $ \bar{b}$ swapped.    
  \end{claim*}     

To prove the first statement in the claim, we have to show that 
$ \bar{a}  \bar{b} =0$ and $ \bar{a} + \bar{b} =1$.
Indeed,
$\bar a \bar  b$ is the product of all the sums of the form 
$ c_{1,f(1)}  + \dots + c_{n,f(n)}$,
with $f: \{ 1, \dots , n\} \to \{ 0, 1  \}$, thus,  by distributivity
\begin{align*} 
0 &= c_{1,0}c_{1,1}+ c_{2,0}c_{2,1}+  \dots+ c_{n,0}c_{n,1}
\\
&=(c_{1,0}+ c_{2,0}c_{2,1}+  \dots+ c_{n,0}c_{n,1})
(c_{1,1}+ c_{2,0}c_{2,1}+  \dots+ c_{n,0}c_{n,1})
\\
&=
(c_{1,0}+ c_{2,0}+  c_{3,0}c_{3,1} + \dots+ c_{n,0}c_{n,1})
(c_{1,0}+ c_{2,1}+  c_{3,0}c_{3,1} +  \dots+ c_{n,0}c_{n,1})
\\
& \phantom{ = \ \ }
(c_{1,1}+ c_{2,0}+  c_{3,0}c_{3,1} + \dots+ c_{n,0}c_{n,1})
(c_{1,1}+ c_{2,1}+  c_{3,0}c_{3,1} +  \dots+ c_{n,0}c_{n,1}) =
 \dots 
\\
& = 
 \prod\{ \,    c_{1,f(1)}+c_{2,f(2)}+   \dots + c_{n,f(n)}  \mid
f: \{ 1, \dots , n\} \to \{ 0, 1  \} \,\}
= \bar a \bar  b . 
  \end{align*}
Moreover, the sum of each factor in the formula defining 
$\bar{b}$ with
each factor in the formula defining 
$\bar{a}$ gives $1$, hence
 $ \bar{a}+ \bar{b}=1 $, again by distributivity.
 
The displayed formulas in the claim then follow by  De Morgan's
laws. For example, by what we have just proved,
$ \bar{a} = \bar{b}' =   
\left( \prod_f    c_{1,f(1)}  + \dots + c_{n,f(n)} \right) ' =
\sum_f    c'_{1,f(1)}   \dots  c'_{n,f(n)}  =
\sum_f    c_{1,1-f(1)}   \dots  c_{n,1-f(n)}$, where $f$  
varies among those function 
$f: \{ 1, \dots , n\} \to \{ 0, 1  \} $ such that
$  f(1) + \dots + f(n)$  is even.
If $n$ is odd, then $  f(1) + \dots + f(n)$  is even
if and only if $  (1 - f(1)) + \dots + (1-f(n))$  is odd,
hence we get the alternative expression for $ \bar{a}$. 
The other cases are treated in a similar way.

Having completed the proof of the claim,
we now notice that $ \bar{a}$ and $c_{1,0}$
are incomparable, since   
$c_{1,1}+c_{2,0}+c_{3,0} + \dots+ c_{n,0}$
is larger than $ \bar{a}$ but not larger than 
$c_{1,0}$; moreover,
$c_{1,1}c_{2,1}c_{3,1}  \dots c_{n,1}$
is smaller than $ \bar{a}$ but not smaller than 
$c_{1,0}$. Similarly, since $n \geq 2$, both $ \bar{a}$
and $ \bar{b}$ are incomparable with
each one of the elements
 $ c_{1,0}, c_{2,0},  
\dots, \allowbreak  c_{n,0}, c_{1,1}, c_{2,1}, \dots, 
 c_{n,1}$.   

Let $\mathbf S$ be the subsemilattice of the 
(join-semilattice reduct) of $\mathbf  C$ generated by
the elements 
$0 $ and $  c_{1,0}, 
\dots, \allowbreak  c_{n,0}, c_{1,1}, \dots, 
 c_{n,1}, \bar b, \bar a$. In particular,
every nonzero element of $\mathbf S$  is 
$\geq$ than at least one of the elements in the above
list. 
Since we have showed that all 
the elements in the list are pairwise incomparable,
they are distinct atoms of $\mathbf S$. 
By construction, the join operation on $\mathbf S$ 
is the restriction of the join operation on $\mathbf  C$;
in particular, the order relation on $\mathbf S$ is
the restriction of the order relation on $\mathbf  C$,
hence there is no notational issue.
In what follows we will frequently use  meet and complementation
 in $\mathbf  C$,
however, these are used only in order to get conclusions speaking just of
$+$ and $\leq$, hence such conclusions hold in $\mathbf S$, as well. 

In $\mathbf S$ set $ c_{1,0} \nd  c_{1,1} $, $ c_{1,1} \nd  c_{1,0} $,
$ c_{2,0} \nd  c_{2,1} $, \dots, \allowbreak
$ c_{n,0} \nd  c_{n,1} $,  \allowbreak
$ c_{n,1} \nd  c_{n,0} $  and let all the other
pairs of $\mathbf S \setminus \{  0\} $ be $\delta$-related;
in particular,   $ \bar b \ddd \bar a$,
since $ \bar b $ and $ \bar a$ are both distinct from
$ c_{1,0},  
\dots, \allowbreak   
 c_{n,1}$.
Since $ c_{1,0},  c_{1,1},   
\dots, \allowbreak   
 c_{n,1}$
are distinct atoms 
of  $\mathbf S$, 
 $\mathbf S$ is
 a contact join-semilattice.  
Then the property \eqref{d2n} fails in $\mathbf S$, 
by taking $a=\bar a$ and  $b=\bar b$.

Since the inclusion is a join-semilattice embedding
from $\mathbf S$ to $\mathbf  C$, we can apply
 Lemma \ref{lemcombis} in order to endow
$\mathbf  C$ with a weak contact relation.
Note that the assumptions from Lemma \ref{lemcombis}(ii)
are satisfied, since the pairs which are not $ \delta $-related
in $\mathbf S$ 
 are  mutual
complements in $\mathbf  C$.  
Hence the inclusion is also a contact embedding.
Thus clause (1) from Theorem \ref{thmb} is satisfied, hence,
by the equivalence of (1) and (3) in Theorem \ref{thmb},
$\mathbf S$ satisfies (D1), and also \eqref{d1+} by \cite[Lemma 2.3]{cs}. 
  
It remains to check that 
(D2$_{m}$) holds in $\mathbf S$,
for every $m<n$.
Clearly, in \eqref{d2n} one may assume that
each pair  $(c_{i,0},  c_{i,1})$ 
never repeats, 
since from repeating pairs we only get repeated 
(or additional) summands
in the relevant sums. Moreover, 
without loss of generality, 
we can assume that 
no $c_{i,j}$ is equal to $0$.
Indeed, assume that, say, $c_{n,0}=0$.
Then the assumptions of \eqref{d2n} are satisfied
if and only if the assumptions of   
  (D2$_{n-1}$) are satisfied, by discarding the pair
$c_{n,0}$, $c_{n,1}$.  

Since in $\mathbf S \setminus \{ 0 \} $ the only 
 $\delta$-unrelated pairs are 
$(c_{1,0},c_{1,1} )$, $(  c_{2,0},c_{2,1}) $, \dots, 
$  ( c_{n,0},c_{n,1})$, 
 by symmetry, without loss of generality 
we can  assume
that the elements in the premise 
of (D2$_{m}$)  are actually the elements
$(c_{1,0},c_{1,1} )$, \dots,
$  ( c_{m,0},c_{m,1})$
introduced in the definition of $\mathbf S$.
In other words, we may assume that the notation
does not clash, by taking here, of course,
$m$ in place of $n$ in \eqref{d2n}.  
So let us assume that, in the above situation,  $a$ and  $b$
are elements of $\mathbf S$ satisfying the
premises of (D2$_{m}$). 

Using the assumptions in (D2$_{m}$),
we get $ab \leq \prod \{ \,    c_{1,f(1)}  + \dots + c_{m,f(m)}  \mid
f: \{ 1, \dots , m\} \to \{ 0, 1  \}\,\}$, hence, 
arguing as in the proof of $ \bar{a} \bar{b}=0 $, we get   
$ab=0$ in $\mathbf  C$, 
hence $b \leq a' $. 
If  $c_{1,0} \leq a$, then 
$b \leq a' \leq c_{1,0}' =c_{1,1}$. 
Moreover, if 
$c_{1,0} < a$, then
$a' < c_{1,0}'$, thus
$b < c_{1,1}$,
but this is impossible if $b \in S$,
unless $b=0$.
In conclusion, if $a,b \in S$ and 
$c_{1,0} \leq a$, then
either $b=0$, or both $a=c_{1,0}$ and $b=c_{1,1}$.
Then the conclusion of
  (D2$_{m}$), that is, $b \nd  a$,  holds in both cases.
The same argument applies if 
$a$ is $\geq$ than one among the elements $ c_{1,1}, c_{2,0},  c_{2,1}  
\dots, \allowbreak   
 c_{n,1}$.

Since 
every nonzero element of $\mathbf S$  is 
$\geq$ than at least one of the elements $ c_{1,0},
\allowbreak  \dots, \allowbreak 
 c_{n,1}, \bar b, \bar a$,
it remains to treat the cases
$a \geq \bar a$ and $a \geq \bar b$.
We first claim that   $c_{1,f(1)}  + \dots + c_{m,f(m)} \geq \bar{a}$,
for no $f: \{ 1, \dots , m\} \to \{ 0, 1  \}$.
Indeed, for each fixed  $f: \{ 1, \dots , m\} \to \{ 0, 1  \}$,
let $g: \{ 1, \dots , n\} \to \{ 0, 1  \}$ be any function
such that $g(i) \neq f(i)$, for $i=1, \dots, m$.
Then   the meet of 
$c_{1,f(1)}  + \dots + c_{m,f(m)}$ and
 $c_{1,g(1)}c_{2,g(2)}  \dots  c_{n,g(n)}$
is $0$. On the other hand, by the Claim above,
and choosing $g(n)$ appropriately
(this can be done, since $n > m$),
we have $c_{1,g(1)}c_{2,g(2)}  \dots  c_{n,g(n)} \leq  \bar a $.
Since we are working in the
Boolean algebra 
freely generated by 
$ c_{1,0},  
\dots, \allowbreak  c_{n,0}$ and 
the $c_{i,1}$'s are their complements, then 
$c_{1,g(1)}c_{2,g(2)}  \dots \allowbreak  c_{n,g(n)} >0$.
This shows that     
$c_{1,f(1)}  + \dots + c_{m,f(m)} $ is not $ \geq \bar{a}$.

In the assumptions of 
(D2$_{m}$), in the nontrivial cases,
we have that $a \leq c_{1,f(1)}  + \dots + c_{m,f(m)}$,
for at least one function  $f: \{ 1, \dots , m\} \to \{ 0, 1  \}$.
Since we have showed that 
$\bar a \nleq c_{1,f(1)}  + \dots + c_{m,f(m)}$,
it is not the case that $ \bar a \leq a$.
Similarly,  $ \bar b \nleq a$.
Hence the only remaining case to be treated
 is the trivial situation when
$b \leq c_{1,f(1)}  + \dots + c_{m,f(m)}$,
for every function  $f: \{ 1, \dots , m\} \to \{ 0, 1  \}$.
 In this case $b=0$, hence the conclusion of 
(D2$_{m}$) follows.
\end{proof}  

The argument in the proof of 
Theorem \ref{dn} 
is very similar in spirit to
\cite[Example 5.2(d)]{cs}.
The main simplification here is obtained by constructing $\mathbf S$
directly as a join-subsemilattice of some Boolean algebra,   
thus Theorem \ref{thmb} can be invoked in order
to get (D1), with no need of complicated computations.
Notice that, while $\mathbf S$ in the proof is comparatively
well-behaved, the weak contact on $\mathbf  C$ is not even additive.
Indeed, $ \bar{a}$ and $ \bar{b}$ are in contact, but they 
are sums of distinct atoms of $\mathbf  C$, by the Claim,
and such atoms are not in contact pairwise.
Of course, this argument can be carried out in $\mathbf  C$
but not in $\mathbf S$.

Having proved Theorem \ref{dn}, the
non-existence of finite axiomatizability
of the class of Ivanova contact join semilattices
 is now an immediate consequence
of the compactness theorem

\begin{corollary} \labbel{nonfin}
The class of  Ivanova contact join-semilattices, namely, the class of
 contact join-semilattices satisfying \textup{(D1)} and \textup{(D2)},
is not first-order finitely axiomatizable.
\end{corollary} 

\begin{proof} 
Conditions (D1) and \eqref{d2n} can be expressed
as first-order sentences, say, $\varphi$  and $\varphi_n$,
thus the theory $T=\{ \varphi \} \cup \{ \, \varphi_n \mid n \in \mathbb N^+ \,\}$
axiomatizes the class of Ivanova contact join-semilattices.
Assume towards a contradiction that
the class is finitely axiomatizable.  So,  there is a 
 sentence $ \psi$ axiomatizing it.
By the completeness theorem,
  $ \psi$ is a consequence of $T$; then, by the compactness
theorem, $\psi$ is also a consequence of some finite subset of $T$,
say,  $\psi$ is  a consequence of
 $\{ \varphi \} \cup \{ \, \varphi_n \mid n \leq \bar n\,\}$,
for some $ \bar{n}$. But all the sentences of  $T$ are  consequences of 
  $\psi$, since $\psi$ axiomatizes the same class, hence 
all the sentences of  $T$ are  consequences of
 $\{ \varphi \} \cup \{ \, \varphi_n \mid n \leq \bar n\,\}$.
This contradicts Theorem \ref{dn}.
\end{proof}  

An anonymous referee suggested that
Corollary \ref{nonfin}  can be strengthened to non-axiomatizability
with finitely many variables. This is indeed the case,
and automatically follows from the fact that semilattices,
hence also weak contact join-semilattices, 
are locally finite. Just observe that, modulo the axioms of
a locally finite theory in a finite language, for every $n$,
up to logical equivalence,
 there are only a finite number of sentences
containing at most $n$ variables.  
Thus, for a locally finite theory in a finite language, 
finite axiomatizability is the same as being axiomatized using only
finitely many variables.

\section{Further remarks} \labbel{fur}

\begin{proposition} \labbel{propba}
If $\mathbf S$ is a join-semilattice with overlap weak contact
and $\mathbf S$ satisfies (D1), then $\mathbf S$ satisfies (D2),
in particular, $\mathbf S$ is additive, by \cite[Remark 2.2]{cs}.
 \end{proposition} 

\begin{proof} 
Let $n$  be an arbitrary positive integer
and let $a^*$  and $b^*$  be two elements of $S$ 
 satisfying the 
antecedent of \eqref{d2n}, namely, for any
 $f: \{ 1, \dots , n\} \to \{ 0, 1  \}$,
 either the former or the latter is $\leq   c_{1,f(1)}  + \dots + c_{n,f(n)}$.
 Thus, if some $d \in S$  is below both $a^*$  and $b^*$,
then, for every $f: \{ 1, \dots , n\} \to \{ 0, 1  \}$,
$d \leq   c_{1,f(1)}  + \dots + c_{n,f(n)}$.
Since by \cite[Lemma 2.3]{cs} $\mathbf S$ 
satisfies \eqref{d1+}, we see that $d \leq 0$  by putting $0$  for $a$ 
and $d$  for $b$   in \eqref{d1+}.
 Thus, by the  arbitrariness of $d$, 
the element $0$ is the only lower bound
of $a^*$  and $b^*$, hence it is  the meet of $a^*$  and $b^*$.
 Since the contact is  overlap by
assumption, we get $a^* \nd b^*$.
\end{proof}    

\begin{corollary} \labbel{ivov}
The class of those Ivanova contact join-semilattices
which have overlap contact is finitely axiomatizable.
 \end{corollary}

 \begin{proof}
By  Theorem \ref{ivan} and Proposition \ref{propba},
 a weak contact join-semilattice $\mathbf S$
with overlap contact 
is Ivanova if and only if $\mathbf S$ 
satisfies (D1).
Then notice that  both (D1) and ``having overlap contact'' 
are properties expressible by a first-order sentence. 
 \end{proof}

\begin{remark} \labbel{final}
We present a final remark
about the classes of contact join-semilattices discussed in this
note.
We have showed that the class of
Ivanova contact Join-semilattices is not finitely
axiomatizable, while the larger class of
 weak contact join-semilattices
satisfying the conditions in Theorem \ref{thmb} is
indeed  finitely axiomatizable. 
 This result might suggest the idea that
the latter class is more natural, and this is surely 
a reasonable point of view.

However, the situation may be seen from a different
perspective. As we hinted in the introduction,
$n$-ary ``hypercontact''   relations are much more general than
binary relations \cite{V}. In \cite{hypercs} we study
classes of join-semilattices with a hypercontact relation,  
parallel to the classes considered here. The main difference is that
in the hypercontact case
finite axiomatizability never occurs. Again, this might be an argument
in favor of considering binary relations only.
On the other hand, one might draw the conclusion 
that considering just binary relations  is an oversimplification.
Three regions might be pairwise in contact without being
in contact.

We are not taking position in favor of one alternative
or the other; rather, we believe that both kinds of structures are interesting
for their own sake and each one has its own specific  uses and applications.
On the other hand, it is surely a significant fact that, in the parallel evolution
of the notion of \emph{event structure} in computer science, a decided shift
occurred from binary relations \cite[Section 8]{WN} to 
 $n$-ary relations \cite[Subsec.\ 2.1.2]{W11}.  
In an even more general applied setting, the advantage of considering
$n$-ary interactions in place of just binary interactions
(e.~g., hypergraphs instead of graphs) is analyzed in \cite{BCI}.
See the introduction of \cite{hypercs}
for further comments and examples. 
 \end{remark}

\begin{acknowledgement}
The author thanks anonymous referees for useful comments.

Work performed under the auspices of G.N.S.A.G.A. 

The author acknowledges the MIUR Department Project awarded to the
Department of Mathematics, University of Rome Tor Vergata, CUP
E83C18000100006.
 \end{acknowledgement}

\end{document}